\documentclass[a4paper]{article}
 \usepackage{amsmath}
 \usepackage{amssymb} 
\usepackage{theorem}
\usepackage{fullpage}

\newtheorem{theorem}{Theorem}

\newtheorem{lemma}{Lemma}
\newtheorem{proposition}{Proposition}

\theorembodyfont{\rm} 

\newtheorem{remark}{Remark}

\newenvironment{proof}{\noindent {\it Proof}.}{\hfill$\Box$}

\newcommand\diam{\operatorname{diam}}

\title{Large entropy implies existence of a maximal entropy measure for interval maps
\footnotetext{Discrete Contin. Dyn. Syst. Ser. A, {\bf 14} (4), 673-688, 2006.}
}

\author{J{\'e}r{\^o}me Buzzi and Sylvie Ruette}
\date{}
\begin{document}
\maketitle

\begin{abstract}
We give a new type of sufficient condition for the existence of
measures with maximal entropy for an interval map $f$, using some non-uniform
hyperbolicity to compensate for a lack of smoothness of $f$. 
More precisely, if the topological entropy of a $C^1$ interval map is
greater than the sum of the local entropy and the entropy of the critical
points, then there exists at least one measure with maximal entropy. As a
corollary, we obtain that any $C^r$ interval map $f$ such that
$h_{{\rm top}}(f)>2\log\|f'\|_{\infty}/r$ possesses measures with maximal entropy.
\end{abstract}

\section{Introduction}

Let $f\colon X\to X$ be a continuous map, where $X$ is a compact metric
space with distance denoted by $d$. The
$(\epsilon,n)$-ball around
$x$  is the set
 $$
     B(x,n,\epsilon) = \{ y\in X\mid  \forall\, 0\leq k<n, \; d(f^ky,f^kx)\leq\epsilon
     \},
 $$
and, if $S\subset X$,  $r(\epsilon,n,S)$ is the minimum number of 
$(\epsilon,n)$-balls the union of which covers $S$.
Recall that the entropy is a measure of dynamical complexity (see \cite{DGS}
for background).
Namely, the topological entropy of $f\colon X\to X$ counts
the number of orbits in the following way, according to Bowen's definition
\cite{Bowen0}:
 $h_{\rm top}(f) = h_{\rm top}(X,f) = \lim_{\epsilon\to 0} h_{\rm top}(X,f,\epsilon)$
 with
 $$
 h_{\rm top}(X,f,\epsilon)=
\limsup_{n\to+\infty}\frac1n\log r(\epsilon,n,X).
 $$
The entropy $h_{\rm top}(S,f)$ of a (not necessarily invariant) subset 
$S\subset X$ is defined in the same way.

The entropy of an invariant and ergodic probability measure $\mu$ of
$f$ is similarly defined, according to Katok's formula \cite{Katok}: 
$h(\mu,f)=\lim_{\epsilon \to 0} h(\mu,f,\epsilon)$, with
 $$
   h(\mu,f,\epsilon) =\limsup_{n\to+\infty} \frac1n\log 
\inf_{\mu(Y)\geq\lambda} r(\epsilon,n,Y)
 $$
where $\lambda$ is any  number  in $(0,1)$.
 
In this continuous and compact setting, it is well-known that the
{\em variational principle} holds (see, e.g., \cite{Walters}): the
topological entropy
$h_{{\rm top}}(f)$ of
$f$ is equal to the supremum of the metric entropies $h(\mu,f)$ taken
over all $f$-invariant probability measures. A measure $\mu$ such that
$h(\mu,f)=h_{{\rm top}}(f)$ is called a {\em maximal measure}. Such
measures, when they exist, are particularly interesting because they
reflect the whole topological complexity of the system, and they enable
to see where this complexity concentrates.

However maximal measures do not always exist. Continuity and even 
mild differentiability are insufficient to ensure their existence in
contrast to the generality of the variational principle. 
In fact, given $r<+\infty$ arbitrary large, there exist $C^r$ 
diffeomorphisms of compact $4$-dimensional
manifolds (constructed by M. Misiurewicz \cite{Misiu1}) as well as $C^r$
interval maps \cite{BuzziSIM,Ruette} (necessarily with an infinite
critical set, see below) which have no maximal measure. 

There are mainly  two types of situation
when existence is known: i) when the dynamics has some
expansiveness (for example if the system is expansive, see e.g.,
\cite{DGS},  or in the case of uniform hyperbolicity, see e.g.,
\cite{Bowen}); ii) when the map is $C^{\infty}$ \cite{Newhouse}.
In both cases, and in fact in all existence results we know of,
one proves that the metric entropy $\mu\mapsto h(\mu,f)$ is upper 
semicontinuous and therefore reaches its supremum by compactness 
(see e.g., \cite{DGS}).  
The only exceptions are the abstract characterizations of existence
due to M. Denker \cite{Denker} in a topological setting and to S. Newhouse
(Theorem 8 of \cite{Newhouse}) for diffeomorphisms. These
results are obtained by establishing upper semicontinuity of the entropy
on an appropriate compact subset of measures.

\medbreak

Our goal is to show that non-uniform
hyperbolicity and finite order differentiability can be combined
to get a criterion of existence of maximal measures.
In this paper, we  focus on continuous interval 
maps $f\colon [0,1]\to[0,1]$ with non-zero topological entropy. 

\medbreak

Let $C(f)$ denote the critical set of $f$, that is, the set of points
which have no neighbourhood on which $f$ is monotonic (if $f$
is $C^1$ then $C(f)$ is contained in the zeroes of $f'$). 
If $C(f)$ is finite, $f$ 
is a (continuous) {\it piecewise
monotonic map} and for such maps existence is known at least since
Hofbauer's paper \cite{Hofbauer}.  The uniqueness of the maximal measure
was first shown for $\beta$-transformations by Takahashi
\cite{Takahashi} (see also \cite{Hofbauer-beta}), then Hofbauer extended
the method -- association of a Markov shift, the 
{\em Markov diagram}, to the initial system -- to general piecewise 
monotonic maps \cite{Hofbauer}. He
showed that these maps have a finite non zero number of ergodic maximal
measures as soon as their topological entropy is positive, and that
the maximal measure is even unique if in
addition $f$ is topologically transitive. 

\medbreak

For continuous interval maps with an infinite critical set the
situation is more complex. Neither existence nor finite multiplicity
of ergodic maximal measures are guaranteed but we are
going to give a sufficient condition in the form of a lower bound on
the topological entropy.

Two quantities play an important role. The first one is the
{\it topological entropy of the critical set}: its
``smallness'' can replace the finiteness of the critical set (which was
the required assumption in Hofbauer's work).
Namely, it was shown in \cite{BuzziThese,BuzziSIM} that 
a Markov diagram can be associated to any interval map $f$ and,
if $h_{\rm top}(f)>h_{\rm top}(C(f),f)$,
then there is a bijection between the maximal measures
of $f$ and those of its Markov diagram (see section \ref{secDiagram}).
This is a decisive step because Gurevich gave an equivalent condition
for existence and uniqueness of maximal measures for transitive
Markov shifts \cite{Gur2}. In many cases the condition 
$h_{\rm top}(f)>h_{\rm top}(C(f),f)$ can be checked using the fact that the
topological entropy of the critical set is bounded 
by $\log\|f'\|_{\infty}/r$ for a $C^r$ interval map $f$ 
(if $r$ is not an integer, this means that $f$ is $C^{[r]}$ and the 
$[r]$-th derivative is $(r-[r])$-H{\"o}lder); in particular 
$h_{{\rm top}}(C(f),f)$ is equal to zero for a $C^{\infty}$ map 
\cite{BuzziThese,BuzziSIM}.
Actually, for an interval map that is $C^{1+\alpha}$, $\alpha>0$,
the condition $h_{{\rm top}}(f)>0$ is enough to have a bijection
between the maximal measures of $f$ and those of a Markov shift provided
one uses a variant of the Markov diagram (see \cite{BuzziEE}).

The second key notion is that of {\em local
entropy}.  The notion of $\epsilon$-local entropy was introduced 
by Bowen \cite{Bow} to bound the difference between $h(\mu,f)$ and
the entropy of a partition with diameter less than $\epsilon$,
then Misiurewicz showed that local entropy
(that he called {\em conditional topological entropy}) bounds the defect in 
upper semicontinuity of the metric entropy $\mu\mapsto h(\mu,f)$ 
\cite{Misiu2}. Hence if the local entropy is zero then 
there exists some maximal measure.

We shall work with the  following (equivalent) definition:

\medbreak\noindent
{\bf Definition }{\it 
The local entropy of a continuous self-map $f \colon X\to X$ of a compact
metric space is 
\footnote{Let us notice that in this definition the supremum over 
all points $x$ can be moved outside of the limits on $\epsilon,\delta$ and
$n$ (see \cite{Bow} or \cite{DGS}).}
$h_{\rm loc}(f) = \lim_{\epsilon\to0} h_{\rm loc}(f,\epsilon)$ where
 $$
     h_{\rm loc}(f,\epsilon) = \lim_{\delta\to0} \limsup_{n\to+\infty} \frac1n\log
        \sup_{x\in X} r(\delta,n,B(x,n,\epsilon)).
 $$
}

\medbreak

The local entropy is bounded by $\frac{d}{r}\log\sup\|f'\|_\infty$  
for a $C^r$ map on a compact manifold of dimension $d$.
This was proved for a (slightly weaker) measure-theoretic 
local entropy by Newhouse \cite{Newhouse} and then for exactly the
above notion by one of us \cite{BuzziSIM}. In particular it is zero for 
$C^\infty$ maps; notice that it is zero for piecewise monotonic maps too. 

\medbreak

\noindent
{\it Remark.} Using Blokh's spectral decomposition \cite{Blokh1,Blokh2} for
a continuous interval map $f$, it was shown in \cite{BuzziThese,BuzziSIM}
that there are only finitely many connected components in the Markov
diagram with entropy close to $h_{\rm top}(f)$ if $h_{\rm top}(f)>h_{\rm loc}(f)$,
each of which supporting at most one ergodic maximal measure
according to Gurevich \cite{Gur2}.
Using the above
bijection between the maximal measures of $f$ and those of the
Markov diagram, this implies that there are only finitely many ergodic
maximal measures if $h_{\rm top}(f)>\max(h_{\rm top}(C(f),f),h_{\rm loc}(f))$
(for $C^{1+\alpha}$ maps, $h_{\rm top}(f)>h_{\rm loc}(f)$ is in fact sufficient 
using \cite{BuzziEE}).

\medbreak
Recalling that $\log\|f'\|_\infty/r$ bounds both $h_{\rm loc}(f)$ and
$h_{\rm top}(C(f),f)$ for a $C^r$ interval map $f$, we see that as soon as $f$ 
satisfies $\log\|f'\|_{\infty}/r< h_{\rm top}(f)$, then there are only finitely 
many ergodic maximal measures. This condition is optimal in the sense that 
there exist $C^r$ interval maps $f$ with infinitely many ergodic maximal
measures satisfying the equality: $h_{\rm top}(f)=\log\|f'\|_\infty/r$ (see
\cite{BuzziSIM}).

\medbreak

The remaining open question therefore is that of existence of maximal
measures for interval maps with infinite critical set and finite smoothness. 
Indeed, existence had only been proved when $h_{\rm loc}(f)=0$, which is known 
to be the case for piecewise monotonic and $C^\infty$ maps.
We give an answer to this question in the $C^1$ case:

\medbreak\noindent
{\bf Theorem~\ref{theo:main-theorem}\ }{\it 
Let $f\colon [0,1]\to [0,1]$ be a $C^1$ map and
$C(f)$ the critical set of $f$.  Assume that
$h_{{\rm top}}(f)>h_{{\rm top}}(C(f),f)+h_{{\rm loc}}(f)$. Then $f$ admits a maximal measure.
Moreover, the number of ergodic maximal measures is finite and, if $f$
is transitive, the maximal measure is unique.
}
\medbreak

Using the previously mentioned bounds on 
local entropy and entropy of the critical set in terms of the
differentiability of the map, we get a condition that is easier to
compute:

\medbreak\noindent
{\bf Corollary }{\it
Let $f\colon [0,1]\to[0,1]$ be a $C^r$ map of the interval with $r\geq 1$.

If $h_{\rm top}(f)>2\log\|f'\|_\infty/r$, then $f$ has a finite non-zero
number of maximal measures.
}
\medbreak

\noindent
{\it Remark.}
This Corollary is relevant only for $r>2$ because $h_{{\rm top}}(f)\leq
\log\|f'\|_{\infty}$ (see, e.g., \cite{DGS}).

\medbreak
The finiteness result in Theorem~\ref{theo:main-theorem} was
already proved in \cite{BuzziThese,BuzziSIM} under weaker hypothesis. 
We nevertheless include it for completeness and also because it is
obtained in a completely different way here, in fact as a slight
variation of the proof of existence.

\medbreak

For interval maps such that $h_{{\rm top}}(C(f),f)=0$,
Theorem~\ref{theo:main-theorem} is sharp: for all $1\leq r<+\infty$ there
exist $C^r$ interval maps that have no maximal measure and such that
$h_{{\rm top}}(C(f),f)=0$ and 
$h_{{\rm top}}(f)=h_{loc}(f)=\frac{1}{r}\log\|f'\|_{\infty}$ (see 
\cite{BuzziSIM, Ruette}). These examples can be adapted to show that
this Theorem indeed applies to maps such that the metric entropy 
$\mu\mapsto h(\mu,f)$ is {\em not upper semicontinuous}, in contrast
to all other existence results for interval maps.
In fact, we get examples with the defect in upper semicontinuity
as large as $h_{\rm loc}(f)=\frac1r\log\|f'\|_\infty$. 

\medbreak

For interval maps with $h_{{\rm top}}(C(f),f)>0$ we do not know 
whether the Theorem is optimal.
Actually, the above Corollary was conjectured, without the factor of $2$ in
\cite{BuzziSIM}. We still believe in that conjecture.
In fact, we make the bolder

\medbreak\noindent
{\bf Conjecture }{\it
If $f\colon [0,1]\to[0,1]$ is a {\emph continuous} interval map such that
$h_{\rm top}(f)>h_{\rm loc}(f)$ then $f$ admits measures with maximal entropy.
}
\medbreak

A way to prove this, would be to establish that for interval maps, if
$\mu_1,\mu_2,\dots$ is a sequence of invariant probability measures
vaguely converging to some $\mu_*$, then:
 $$
       \limsup_{n\to\infty} h(\mu_n,f) \leq \max(h(\mu_*,f),h_{\rm loc}(f)).
 $$
That is, the obvious bound with a {\it sum} could be replaced for interval
maps by a {\it maximum} (this is obviously false in higher dimensions).

\medbreak\noindent
{\it Remark.}
For interval maps with $h_{\rm top}(C(f),f)=h_{\rm top}(f)$ which
are not $C^{1+\alpha}$, $\alpha>0$, the relevant dynamics may be
completely missed  by the Markov diagram. Hence proof of the above
conjecture in its full generality probably requires a different method 
from the one used in this paper.

\subsection*{Outline of the paper}
We begin by recalling the relevant theory of countable Markov shifts.
In the second section, we introduce the Markov diagram, i.e., a countable
Markov shift representing the interval map. In the third section
we prove that measures escaping to infinity in this Markov diagram
have small entropy. Finally we deduce the Main Theorem
from the previous results.

\section{Background on Markov shifts}\label{sec:Markov}
\subsection{Graphs and Markov shifts}
Let $G$ be an oriented graph with a countable set of vertices. 
If $u,v$ are two vertices,
there is at most one arrow $u\to v$. A {\em path} of length $n$
is a sequence of vertices
$(u_0,\cdots,u_n)$ such that $u_i\to u_{i+1}$ is an arrow 
in $G$ for $0\leq i<n$. This path is called a {\em loop} if $u_0=u_n$.
The graph $G$ is called {\em strongly connected},
if for all vertices $u,v$ there exists a path in $G$ from $u$ to $v$.
A {\em connected component} $G'$ is a strongly connected subgraph which
is maximal for inclusion;
two connected components are equal or disjoint.

Let $u$ be a vertex. In \cite{VJ} Vere-Jones defines the following 
quantities.
\begin{itemize}
\item $p_{u}^G(n)$ is the number of loops $(u_0,\cdots,u_n)$ 
such that $u_0=u_n=u$; $R_{u}(G)$ is the radius of convergence
of the series $\sum p_{u}^G(n)z^n$.
\item $f_{u}^G(n)$ is the number of loops $(u_0,\cdots,u_n)$ 
such that $u_0=u_n=u$ and $u_i\not=u$ for $0<i<n$; 
$L_{u}(G)$ is the radius of convergence
of the series $\sum f_{u}^G(n)z^n$.
\end{itemize}

If $G$ is strongly connected, then $R_{u}(G)$ does not depend on $u$; 
in this case it is denoted by $R(G)$.

\medbreak
Let $G$ be an oriented graph.
$\Sigma_+(G)$ is the set of one-sided infinite paths in $G$, that is,
$$
\Sigma_+(G)=\{(v_n)_{n\in{\mathbb N}}\mid \forall n\in {\mathbb N}, v_n\to v_{n+1} 
\mbox{ in } G \}.
$$
$\sigma$ is the shift on $\Sigma_+(G)$, $\sigma((v_n)_{n\in{\mathbb N}})=
(v_{n+1})_{n\in{\mathbb N}}$. The {\em
Markov shift}  on the graph $G$ is the system $(\Sigma_+(G),\sigma)$.

The set $G$ is endowed with the discrete topology and $\Sigma_+(G)$
is endowed with the induced topology of $G^{{\mathbb N}}$, which has the product
topology. The space $\Sigma_+(G)$ is not compact unless $G$ is finite.
The system $(\Sigma_+(G),\sigma)$ is transitive if and only if
the graph $G$ is strongly connected. 

If $S\subset G$, the {\em cylinder} $[S]$ is defined as
$$
[S]=\{(u_n)_{n\in{\mathbb N}} \in \Sigma_+(G)\mid u_0\in S\}.
$$

\subsection{Entropy and maximal measures}
If $G$ is an oriented graph, the {\em Gurevich entropy} of $G$ is
defined as
$$
h(G)=\sup_{u\in G}-\log R_{u}(G).
$$
If $G'$ is a connected component of $G$, then $R_u(G)=R(G')$ for all
$u\in G'$, hence 
$$
h(G)=\sup\{ h(G')\mid G' \mbox{ connected component of } G\}.
$$

Moreover, the variational principle is still valid for the Gurevich
entropy.

\begin{theorem}[Gurevich \cite{Gur1}]\label{theo:variational-principle}
Let $G$ be an oriented graph. Then
$$
h(G)=\sup\{h(\mu,\sigma)\mid \mu
\ \sigma\mbox{-invariant probability measure on } \Sigma_+(G)\}.
$$
Moreover, the supremum can be taken on ergodic Markov measures only.
\end{theorem}

A {\em maximal measure} is a $\sigma$-invariant
probability measure $\mu$ on $\Sigma_+(G)$ whose 
entropy is maximal, that is, $h(\mu,\sigma)=h(G)$.

An ergodic measure on $\Sigma_+(G)$ is necessarily supported by some
$\Sigma_+(G')$, where $G'$ is a connected component of $G$. Therefore
an ergodic maximal measure on $\Sigma_+(G)$ is a maximal measure for
a connected component $G'$ with $h(G')=h(G)$.

\newcommand{\lignelarge}{\lower1.5ex\hbox{\rule{0ex}{4ex}}}

\begin{table}[ht] 
\centerline{
\begin{tabular}{l|c|c|c}
                     & transient   & null      & positive  \\
                     &             & recurrent & recurrent \\
\hline
\lignelarge
$\displaystyle\sum_{n>0} f^G_{v}(n)R(G)^n$  & $<1$        & $1$       & $1$     \\
\hline
\lignelarge
$\displaystyle\sum_{n>0} nf^G_{v}(n)R(G)^n$ &$\leq+\infty$&$+\infty$  &$<+\infty$ \\
\end{tabular}
}
\caption{classification of strongly connected graphs into
transient, null recurrent and positive recurrent graphs 
(it does not depend on the vertex $v$).
\label{tab:classification}}
\end{table}

A strongly connected oriented graph $G$ is called {\em transient, null
recurrent} or {\em positive recurrent} depending on the values of the
series $\sum f_{v}^G(n)z^n$ and its derivative at point $z=R(G)$
(see Table \ref{tab:classification}). This classification is due
to Vere-Jones \cite{VJ}. In \cite{Gur2} Gurevich shows that, if $G$ is 
strongly connected, the Markov shift $(\Sigma_+(G),
\sigma)$ admits a maximal measure if and only if $G$ is positive recurrent,
and in this case this measure is unique and it is an ergodic 
Markov measure. 

We sum up the results above in the next Theorem.

\begin{theorem}\label{theo:existence-Markov}
Let $G$ be an oriented graph. 
\begin{enumerate}
\item If $\nu$ is an ergodic maximal measure on $\Sigma_+(G)$,
then $\nu$ is supported by a connected component of maximal entropy which is
positive recurrent. 
\item If $G$ is strongly connected then it admits at most one maximal
measure, and when it exists it is an ergodic Markov measure.
\end{enumerate}
\end{theorem}

\begin{remark}
Two-sided infinite paths (i.e. paths indexed by ${\mathbb Z}$) are often considered
instead of one-sided infinite paths. Gurevich stated his
results for such invertible Markov shifts. However they are still valid
in the non-invertible case that interests us.
\end{remark}

\subsection{Almost maximal measures escaping to infinity}
\label{subsecMS}

Let $G$ be an oriented graph and $G\cup\{\infty\}$ its one-point
compactification. The set $\overline{\Sigma_+(G)}\subset 
(G\cup\{\infty\})^{{\mathbb N}}$ is compact and so is the set of $\sigma$-invariant
measures on $\overline{\Sigma_+(G)}$ \cite{DGS}. Gurevich and Savchenko
showed that if $G$ is either transient or null recurrent then any
sequence of ergodic measures $(\nu_n)_{n\geq 1}$ whose entropy tends to
$h(G)$ converges to the Dirac measure $\delta_{\infty}$ on 
$\overline{\Sigma_+(G)}$ (for the weak-* topology). This is
Theorem~6.3(1) in \cite{GS} for a null potential, we restate the
measures convergence in term of cylinders then we generalise this result 
to all oriented graphs with no maximal measure.

\begin{theorem}[Gurevich-Savchenko \cite{GS}]\label{theo:GS}
Let $G$ be a strongly connected graph of finite entropy which is
not positive recurrent. If $(\nu_n)_{n\geq 1}$ is a sequence of
ergodic measures such that $\displaystyle\lim_{n\to+\infty}h(\nu_n,\sigma)
=h(G)$ then for all finite subsets of vertices $F$ one has
$\displaystyle\lim_{n\to+\infty}\nu_n([F])=0$.
\end{theorem}

\begin{proposition}\label{prop:sequence-of-measures}
Let $G$ be an oriented graph of finite entropy.
Suppose that $\Sigma_+(G)$ admits no maximal measure. Then 
there exists a sequence
of ergodic Markov measures $(\nu_n)_{n\geq 1}$ such that
$\displaystyle\lim_{n\to+\infty}h(\nu_n,\sigma)=h(G)$ and for all finite subsets of
vertices $F$, $\displaystyle\lim_{n\to+\infty}\nu_n([F])=0$.
\end{proposition}

\begin{proof}
Suppose first that $G$ has a connected component $G'$ with $h(G)=h(G')$.
By Theorem \ref{theo:variational-principle} there exists a sequence of
ergodic Markov measures $(\nu_n)_{n\geq 1}$ on $\Sigma_+(G')$ such that
$\lim_{n\to+\infty}h(\nu_n,\sigma)=h(G')$. The measures $\nu_n$
can be seen as measures on $\Sigma_+(G)$. By assumption $\Sigma_+(G')$ 
admits no maximal measure thus $G'$ is not positive recurrent by
Theorem \ref{theo:existence-Markov} and Theorem \ref{theo:GS}
applies.

\medbreak
Now suppose inversely that $G$ has no connected component of entropy 
equal to $h(G)$.
This assumption implies that there exists a sequence of distinct connected
components $(G_n)_{n\geq 0}$ such that $\lim_{n\to+\infty}h(G_n)=h(G)$.
According to Theorem \ref{theo:variational-principle}
there exists an ergodic Markov measure $\nu_k$ on $\Sigma_+(G_{k})$
such that $h(\nu_k,\sigma)\geq h(G_{k})-\frac{1}{k}$. This implies that
$\lim_{k\to+\infty}h(\nu_k,\sigma)=h(G)$. Moreover, if $F$ is a finite
subset of vertices of $G$, there exists $n$ such that $F\cap \bigcup_{k\geq n}
G_k=\emptyset$, thus $\nu_k([F])=0$ for all $k\geq n$.
\end{proof}

\begin{proposition}\label{prop:sequence-of-measures2}
Let $G$ be an oriented graph of finite non-zero entropy.
Suppose that $(\nu_k)_{k\geq 1}$ is a sequence of distinct ergodic maximal
measures for $\Sigma_+(G)$. Then for all finite subsets of vertices $F$, 
one has $\displaystyle\lim_{n\to+\infty}\nu_n([F])=0$.
\end{proposition}

\begin{proof}
By Theorem \ref{theo:existence-Markov}, $\nu_n$ is supported by a
connected component $G_n$ and all the graphs $G_n$ are disjoint. Let
$F$ be a finite subset of vertices. There exists an integer $N$ such that
$F\cap G_n=\emptyset$ for all $n\geq N$, thus $\nu_n([F])=0$ for
all $n\geq N$.
\end{proof}

\section{The Markov diagram} \label{secDiagram}

This section is devoted to the reduction of the map on the interval
to a Markov shift. 

This reduction was introduced by Hofbauer \cite{Hofbauer} for piecewise 
monotonic maps (see also Takahashi for a special case \cite{Takahashi}). 
We need the variant introduced in
\cite{BuzziThese,BuzziSIM} for general interval maps. Let us recall its
definition.

Consider $f\colon [0,1]\to [0,1]$ a $C^1$ map and let $C(f)$ be the 
critical set of $f$, that is, the set of points in a neighbourhood
of which $f$ is not monotonic. Let $C_*$ be a finite subset of $[0,1]$ and
$C=C(f)\cup C_*$. The additional set $C_*$ will be needed in the proof of
Theorem \ref{theo:eps-delta}. It does not change anything to the
construction and does not affect the entropy of the critical set.
Indeed,
\begin{equation}\label{eq:htopC}
    h_{{\rm top}}(C,f)=\max(h_{{\rm top}}(C(f),f),h_{\rm top}(C_*,f)) = h_{\rm top}(C(f),f).
\end{equation}

Let ${\mathcal P}$ be the collection of the connected components of $[0,1]\setminus C$
and let ${\mathcal P}^*$ be the set of finite sequences $A_{-n}\dots A_0$, where
$A_i\in{\mathcal P}$.

The set $[A_0\dots A_n]_f$ is  defined as
$$
[A_0\dots A_n]_f=\{x\in[0,1]\mid f^i(x)\in A_i, 0\leq i\leq n\}
=\bigcap_{i=0}^n f^{-i}(A_i).
$$

\begin{lemma}\label{lem:sig-part}
Observe that:
\begin{itemize}
\item
$[A_0\dots A_n]_f$ is an open interval.
\item  $f^n$ restricted to $\displaystyle \overline{[A_0\dots A_n]_f}$ 
is a homeomorphism on its image.
\item $\displaystyle \overline{[A_0\dots A_n]_f}=\bigcap_{i=0}^n 
f^{-i}(\overline{A_i})$ if $[A_0\dots A_n]_f\ne\emptyset$.
\end{itemize}
\end{lemma}

Say that $A_{-n}\dots A_0$ and
$B_{-m}\dots B_0$ are equivalent  if and only if there exists $0\leq k\leq
\min(n,m)$ such that:
 \begin{gather*}
       A_{-k}\dots A_0 = B_{-k} \dots B_0 \\
    f^k([A_{-k}\dots A_0]_f) = f^n([A_{-n} \dots A_0]_f) \\
    f^k([B_{-k}\dots B_0]_f) = f^m([B_{-m} \dots B_0]_f).
 \end{gather*}
We write in this situation $A_{-n}\dots A_0 \approx B_{-m}\dots B_0$.

If $k$ is minimal with the properties above, then $A_{-k}\dots A_0$
is called the {\em significant part} of $A_{-n}\dots A_0$. Two elements of
${\mathcal P}^*$ are equivalent if and only if they have the same significant part.
If $\alpha$ is the equivalence class of $A_{-n}\dots A_0$, we define 
$$
\left<\alpha\right>=f^n([A_{-n}\dots A_0]_f)=\bigcap_{i=0}^n f^i(A_{-i}).
$$

\medbreak
Let ${\mathcal D}$ be the set of the equivalence classes $\alpha\in{\mathcal P}^*/\approx$ with 
$\left<\alpha\right>\not=\emptyset$. If $\alpha,\beta\in{\mathcal D}$, there is an
arrow $\alpha\to\beta$ if and only if there exist $A_{-n},\dots,A_0,A_1
\in{\mathcal P}$ such that $\alpha$ is the equivalence class of $A_{-n}\dots A_0$
and $\beta$ is that of $A_{-n}\dots A_0A_1$. The countable oriented graph
${\mathcal D}$ is called the {\em Markov diagram} associated to $f$ with respect
to $C$. It defines a Markov shift $(\Sigma_+({\mathcal D}),\sigma)$ (see
Section \ref{sec:Markov}).

It is convenient to let ${\mathcal D}_n$ be the collection of equivalence
 classes generated by words of length at most $n+1$. We say
that an element $D$ of ${\mathcal D}_n\setminus{\mathcal D}_{n-1}$ has {\em level or height}
$H(D)=n$.

For $\alpha=(\alpha_n)_{n\geq 0}\in\Sigma_+({\mathcal D})$, let $A_n$ be the element
of ${\mathcal P}$ containing $\left<\alpha_n\right>$. The
sequence $A$ is the {\em projection} or the {\em itinerary} of $\alpha$.
Define
 $$
 \pi(\alpha)\in\bigcap_{n\geq0} \overline{[A_0\dots A_n]_f}
=\bigcap_{n\geq 0} f^{-n}(\overline{A_n}).
 $$

There is an arbitrary choice involved in the definition of $\pi(\alpha)$
when this intersection is a non-trivial interval.
Notice that this occurs only for countably many $\alpha$'s.

If $\nu$ is an atomless $\sigma$-invariant probability 
measure on $\Sigma_+({\mathcal D})$
then $\mu=\pi_*(\nu)$ is a $f$-invariant probability measure on $[0,1]$,
defined by $\mu(B)=\nu(\pi^{-1} B)$. Moreover, $\mu$ is ergodic if
$\nu$ is ergodic.

\begin{theorem}\cite{BuzziSIM}\label{theo:bijection}
Let $f\colon [0,1]\to[0,1]$ be a $C^1$ map that satisfies 
$h_{\rm top}(f)>h_{\rm top}(C,f)$ and let $\Sigma_+({\mathcal D})$, $\pi$ be defined as above.

Then the map $\nu\mapsto \mu=\pi_*(\nu)$ is
a bijection preserving entropy between the $\sigma$-ergodic measures $\nu$
and the $f$-ergodic measures $\mu$ such that $h(\nu,\sigma)>h_{{\rm top}}(C,f)$
and $h(\mu,f)>h_{{\rm top}}(C,f)$.

In particular, $h({\mathcal D})=h_{{\rm top}}(f)$ and $\pi$ induces a bijection between the
maximal measures  of $f$ and $\Sigma_+({\mathcal D})$.
\end{theorem}

We shall need the following facts:

\begin{lemma}\label{lem:follower} 
If $\alpha_0\dots\alpha_n$ is a path on ${\mathcal D}$ and if $A_k$ is the
element of ${\mathcal P}$ containing $\left<\alpha_k\right>$, then
 $$
      \alpha_n \text{ is equivalent to } B_{-m}\dots B_0 A_1\dots A_n
 $$
for any $B_{-m}\dots B_0$ which is equivalent to $\alpha_0$.
\end{lemma}

This is a rephrasing of Lemma 5.4 of \cite{BuzziSIM}. We give a proof 
for completeness.

\begin{proof}
Suppose that $B_{-m}\dots B_0$ is the significant part of $\alpha_0$.
Since $\alpha_0\to \alpha_1$, there exist $A_{-k},\dots,A_0,A_1$ 
in ${\mathcal P}$ such that $\alpha_0$ 
is equivalent to $A_{-k}\dots A_0$ and $\alpha_1$ is equivalent to
$A_{-k}\dots A_0A_1$. Thus, $A_{-k}\dots A_0\approx B_{-m}\dots
B_0$. This implies $k\geq m$, $A_{-m}\dots A_0 = B_{-m}\dots B_0$ and:
 $$
\left<\alpha_0\right>=f^k([A_{-k}\dots A_0]_f)=f^m([B_{-m}\dots B_0]_f).
 $$
It follows immediately that $A_{-m}\dots A_0A_1 = B_{-m}\dots B_0A_1$ and:
\begin{eqnarray*}
f^{m+1}([B_{-m}\dots B_0A_1]_f)&=&
A_1\cap\bigcap_{i=0}^m f^{i+1}(B_{-i})
=A_1\cap f(\left<\alpha_0\right>)\\
&=& f^{k+1}([A_{-k}\dots A_0A_1]_f).
\end{eqnarray*}
i.e., $A_{-k}\dots A_0A_1\approx B_{-m}\dots B_0A_1$.
Moreover $\left<\alpha_1\right>\subset A_1$.

The rest of the proof follows by induction.
\end{proof}

\begin{lemma} \label{lemPasse}
Let $\alpha=(\alpha_n)_{n\geq 0}\in\Sigma_+(D)$ and $x=\pi(\alpha)$. 
If the significant part of $\alpha_n$ is $A_{-k}\dots A_0$ and if
$k\leq n$, then $f^{n-j}x\in \overline{A_{-j}}$ for $0\leq j\leq k$.
\end{lemma}

\begin{proof}
Let $0\leq j\leq k$.
If $\alpha_{n-j}$ is the equivalence class of some $B_{-q}\dots B_0$ then
there exist $B_1,\dots,B_j\in{\mathcal P}$ such that $\alpha_n$ is the equivalence
class of $B_{-q}\dots B_0B_1\dots B_j$ (see Lemma \ref{lem:follower}).
Therefore $B_0\dots B_j=A_{-j}\dots A_0$. 
By definition of $\pi$, this implies $f^{n-j}(x)\in \overline{A_{-j}}$
and proves the Lemma.
\end{proof}

\begin{remark}
Lemma~\ref{lemPasse} would be {\em false} 
if we had used Hofbauer's Markov diagram. Indeed, in Hofbauer's Markov
diagram, the vertices of the graph are not the sequences $\alpha\in {\mathcal D}$
as above but the intervals
$\left< \alpha \right>$. But completely different words $\alpha$ (sharing
only their last symbol) may by coincidence give the same interval. These
words will give disjoint paths ending at the same vertex, in
contradiction with the Lemma.
\end{remark}

Finally, we need that the transitivity of $f$ implies that the
Markov diagram is essentially irreducible.

\begin{lemma} \label{lem:transitive}
If $f$ is  transitive then its Markov diagram contains at
most one connected component with entropy larger than  $h_{\rm top}(C,f)$.
\end{lemma}

\begin{proof}
Let $G_1,G_2\subset{\mathcal D}$ be two connected components with entropy
larger than $h_{\rm top}(C,f)$. By symmetry,
it is enough to build a path from $G_1\to G_2$ to prove that $G_1=G_2$.

Define ${\mathcal P}^n$ as the collection of disjoint open intervals 
$[A_0\dots A_{n-1}]_f$ with $A_i\in{\mathcal P}$. If $x\in [0,1]$, let ${\mathcal P}^n(x)$ 
denote the element of ${\mathcal P}^n$ that contains $x$ when such an element exists.

Let $\alpha_0\in G_1$. Let $I$ be the open, non-empty interval
$\left<\alpha_0\right>$. The set $K=\bigcup_{n\geq0} f^n(I)$ 
is a union of intervals.
By transitivity, $f^k(I)\cap I\ne\emptyset$ for some $k$
so that $K$ is a finite union of intervals. Again by 
transitivity, $K$ is dense in $[0,1]$. Hence $[0,1]\setminus K$ is
reduced to finitely many points.

Fix $\nu_2$ an ergodic and invariant probability measure on
$\Sigma_+(G_2)$ with $h(\nu_2,\sigma)>h_{\rm top}(C,f)$. Let
$\mu_2=\pi_*(\nu_2)$. Let us observe a number of generic properties:
\begin{itemize}
\item $\mu_2$ is non-atomic so that $\mu_2(K)=1$.
\item $\mu_2(\pi (\Sigma_+(G_2)))=1$.
\item $\mu_2(\bigcup_{n,m\geq0} f^{-n}f^m C) = 0$. Otherwise 
$\mu_2(f^m C)=\mu_2(f^{-n}f^m C)>0$ for some $n,m\geq0$ but this
would imply $h(\mu_2,f)\leq h_{\rm top}(f^mC,f)=h_{\rm top}(C,f)$. But
$h(\mu_2,f)=h(\nu_2,f)$ by Theorem \ref{theo:bijection}, 
which leads to a contradiction.
\item for $\mu_2$-a.e. $x$, ${\mathcal P}^n(x)$ is well-defined for all
$n\geq1$ and $\lim_{n\to +\infty} \diam {\mathcal P}^n(x)=0$.
\end{itemize}

 From the properties above we deduce that there exists 
$y\in K\cap\pi(\Sigma_+(G_2))$ such that 
$y\not\in\bigcup_{n,m\geq0} f^{-n}f^m C$ and
$\lim_{n\to +\infty}\diam{\mathcal P}^n(y)=0$. Let $\beta\in \Sigma_+(G_2)$ such that
$y=\pi(\beta)$,  $\beta_0$ being the equivalence class of some 
$B_{-q}\dots B_0$. Define $J=\left<\beta_0 \right>$; this is an open 
interval containing $y$. Since $y\in K$, there exist
$x\in I$ and $k\geq 0$ such that $y=f^k(x)$. 
Moreover, for all $n\geq 0$ there exists $A_n\in{\mathcal P}$ such that
$f^n(x)\in A_n$. Let
$\alpha_n$ be the equivalence class of $A_{-p}\dots A_n$ for all
$n\geq 0$, where $A_{-p}\dots A_0\approx\alpha_0$.
The set $J'=\left<\alpha_k\right>$ is an open interval containing $y$.

By Lemma \ref{lem:follower}, $\beta_n$ is the equivalence 
class of $B_{-q}\dots B_0A_{k+1}\dots A_{k+n}$, with $B_0=A_k$. One has 
${\mathcal P}^n(y)=[A_k\dots A_{k+n-1}]_f$, and its diameter tends to $0$ by
the choice of $y$. Therefore there exists $n\geq 0$ such that
$[A_k\dots A_{k+n}]_f\subset J'\cap J$.
One has
\begin{eqnarray*}
\left<\alpha_{n+k}\right>&=&f^{n+k+p}([A_{-p}\dots A_{n+k}]_f)\\
&=& f^n(f^{k+p}([A_{-p}\dots A_k]_f)\cap [A_k\dots A_{n+k}]_f)\\
&=& f^n(J'\cap [A_k\dots A_{n+k}]_f)\\
&=& f^n([A_k\dots A_{n+k}]_f).
\end{eqnarray*}
The same computation gives
$$ 
\left<\beta_n\right>=f^n(J\cap [A_k\dots A_{n+k}]_f)= 
f^n([A_k\dots A_{n+k}]_f).
$$
Therefore $\alpha_{n+k}= \beta_n$, and $\alpha_0\to\dots\to 
\alpha_{n+k}$ is a path between $\alpha_0\in G_1$ and $\beta_n\in G_2$.
\end{proof}

\section{Entropy at infinity in ${\mathcal D}$}

In this section, we consider a sequence of ergodic measures on
$\Sigma_+({\mathcal D})$ which charge less and less any finite set of vertices
and whose entropy is bounded from below by $h_{{\rm top}}(C(f),f)$. 
We prove (Proposition \ref{prop:high-levels}) that these measures escape to
the high levels of the diagram. Then we show in Theorem \ref{theo:eps-delta}
that, if $C_*=\{k\delta\mid k=1,\ldots, [\delta^{-1}]\}$ for
small
$\delta>0$, such a sequence of measures cannot have a large entropy.

\medbreak
To prove the Proposition we need the following result (which we restrict
to  interval maps and ergodic measures).

\begin{theorem}[Ruelle-Margulis inequality \cite{Ruelle}]\label{theo:RM}
Let $f\colon [0,1]\to [0,1]$ be a $C^1$ map and $\mu$ a $f$-ergodic
measure. The quantity
$$
\lambda(x)=\lim_{n\to+\infty}\frac{1}{n}\log|(f^n)'(x)|
=\lim_{n\to+\infty}\frac{1}{n}\sum_{k=0}^{n-1}\log|f'(f^k(x))|
$$
exists almost everywhere in $[-\infty,+\infty)$ and is almost constant; 
let $\lambda$ be this constant.

Then $h(\mu,f)\leq \max (\lambda,0)$.
\end{theorem}

\begin{proposition}\label{prop:high-levels}
Let $f\colon [0,1]\to [0,1]$ be a $C^1$ map of the interval. Let
$C_*$ be a finite subset of $[0,1]$ and consider the Markov diagram ${\mathcal D}$
associated to $f$ with respect to
$C=C(f)\cup C_*$. Let $(\nu_m)_{m\geq 1}$ be a sequence of invariant, 
ergodic measures on $\Sigma_+({\mathcal D})$ 
such that $h(\nu_m,\sigma)>h_{\rm top}(C(f),f)$ 
and suppose that for all finite subsets 
$F\subset {\mathcal D}$, $\lim_{m\to+\infty} \nu_m([F])=0$. 
Then for all integers $N$, one has
$\lim_{m\to+\infty} \nu_m([{\mathcal D}_N])=0$.
\end{proposition}

Let us remark that in the cases that are of interest to us, the sets
${\mathcal D}_N$ are not finite.

\medbreak
\begin{proof}
Fix an integer $N$. If $r$ is a positive number, we define the following
subset of the Markov diagram:
$$
{\mathcal F}_r=\{ A_{-n}\dots A_0 \in {\mathcal D} \mid n\leq N \mbox{ and }
     \diam\, A_{-k}>r \text{ for all } 0\leq k\leq n\}
\subset {\mathcal D}_N.
$$
The set ${\mathcal F}_r$ is finite because only finitely many elements $A\in{\mathcal P}$
satisfy $\diam\, A>r$. Therefore $\lim_{m\to+\infty}\nu_m([{\mathcal F}_r])=0$
by assumption. 
By definition, $C=C(f)\cup C_*$, where $C_*$ is a finite set,
and $C(f)\subset (f')^{-1}\{0\}$ 
thus there exists $r_0>0$ such that for all $r\leq r_0$ and $A\in{\mathcal P}$,
$$
    \diam\, A\leq r \Rightarrow \forall x\in \overline{A},
d(x,(f')^{-1}\{0\})\leq r.
$$
Let $\epsilon>0$. Fix $0<\beta<1$ such that 
$\frac{\log \|f'\|_{\infty}}{|\log \beta|}<\frac{\epsilon}{N+1}$.
By continuity of $f'$ one can choose $r>0$ such that for all $A\in{\mathcal P}$
with $\diam A\leq r$,
\begin{equation}\label{eq:small-derivative}
      \forall x\in \bar A, |f'(x)|<\beta.
\end{equation}

Choose $m_0$ such that for all $m\geq m_0$, $\nu_m([{\mathcal F}_r])<\epsilon$ and
put $\mu_m=\pi_*(\nu_m)$; $\mu_m$ is ergodic and, according to Theorem 
\ref{theo:bijection}, $h(\mu_m,f)=h(\nu_m,\sigma)>0$.
Let $\lambda(x)=\lim_{n\to+\infty}\frac{1}{n}\sum_{k=0}^{n-1} 
\log|f'(f^k(x))|$. Applying Theorem \ref{theo:RM} we get that
$0<h(\mu_m,f)\leq \lambda(x)$ for $\mu_m$-a.e. $x$.
Consequently, there exists a $\nu_m$-generic point $\alpha=
(\alpha_n)_{n\geq 0}\in \Sigma_+({\mathcal D})$ such that, for $x=\pi(\alpha)$,
$$
\lim_{n\to+\infty}\frac{1}{n}\sum_{k=0}^{n-1} \log|f'(f^k(x))|>0.
$$

Let $n$ be large enough so that 
$\sum_{k=0}^{n-1} \log|f'(f^k(x))|>0$.
Let 
$$
J=\{0\leq k<n\mid 
f^k(x)\in \overline{A} \text{ with }  A\in{\mathcal P} \text{ and }
\diam\, A\leq r \}.
$$ 
Using (\ref{eq:small-derivative}),
one has
\begin{eqnarray*}
0<\sum_{k=0}^{n-1} \log|f'(f^k(x))| & = &
\sum_{k\in J}\log|f'(f^k(x))|+\sum_{k\in [0,n)\setminus J}\log|f'(f^k(x))|\\
 &\leq & -\# J\cdot |\log\beta| +n\log \|f'\|_{\infty}.\\
\end{eqnarray*}
Thus 
$$
      \# J<\frac{n\log \|f'\|_{\infty}}{|\log\beta|}
<\frac{n\epsilon}{N+1}.
$$

Let $N\leq k<n$ be such that $\alpha_k\in {\mathcal D}_N\setminus {\mathcal F}_r$, i.e., the
significant part of $\alpha_k$ is of the form $A_{-q}\dots A_0$ with
$0\leq q\leq N$ with $\diam A_{-p}\leq r$ for some $0\leq p\leq q$. 
Since $p\leq k$, Lemma \ref{lemPasse}
applies and $f^{k-p}(x)\in \overline{A_{-p}}$, thus $k-p\in J$.
Observe that for a given $k$ there are at most $N+1$ indices $p$ as above, thus
$$
\frac{1}{N+1}\#\{N\leq k<n\mid \alpha_k\in{\mathcal D}_N\setminus{\mathcal F}_r\}\leq \# J.
$$
This implies that
\begin{eqnarray*}
\frac{1}{n}\#\{0\leq k<n\mid \alpha_k\in{\mathcal D}_N\setminus{\mathcal F}_r\}&\leq &
\frac{N}{n}+\frac{(N+1)\# J}{n}\\
 &\leq & \frac{N}{n}+\epsilon
\end{eqnarray*}
and this inequality is valid for all integers $n$ large enough.
Moreover, the point $\alpha$ is generic for $\nu_m$, therefore
$$
\nu_m([{\mathcal D}_N\setminus {\mathcal F}_r])=\lim_{n\to+\infty}\frac{1}{n}
\#\{0\leq k<n\mid \alpha_k\in{\mathcal D}_N\setminus {\mathcal F}_r\}\leq \epsilon.
$$
For $m\geq m_0$, one obtains that
$\nu_m([{\mathcal D}_N])=\nu_m([{\mathcal D}_N\setminus {\mathcal F}_r])+\nu_m([{\mathcal F}_r])\leq 2\epsilon$.
This concludes the proof.
\end{proof}

We now turn to the

\begin{theorem}\label{theo:eps-delta}
Let $f\colon [0,1]\to [0,1]$ be a $C^1$ map. 
Let $\gamma>0$. Then there exists $\delta>0$ satisfying the
following property. 
Define $C_*=\left\{k\delta\mid k=1,\dots, [\delta^{-1}]\right\}$
and consider the Markov diagram ${\mathcal D}$ associated to $f$ with respect to
$C=C(f)\cup C_*$.

Let $(\nu_m)_{m\geq 1}$ be a sequence of ergodic measures on 
$\Sigma_+({\mathcal D})$ such that, for all finite subsets $F\subset {\mathcal D}$ one has 
$$
\lim_{m\to+\infty} \nu_m([F])=0.
$$
Then 
$$
\limsup_{m\to+\infty}h(\nu_m,\sigma)\leq h_{{\rm top}}(C(f),f)+h_{{\rm loc}}(f)+\gamma.
$$
\end{theorem}

For the proof of this Theorem, we need two more facts.

The first is a standard estimate. It derives from the Stirling formula.

\begin{lemma}\label{lem:combinatorial}
Let $0<\alpha <1/2$ and $\epsilon>0$. Define
$\phi(\alpha)=-\alpha\log\alpha-(1-\alpha)\log(1-\alpha)$.
Then for all integers $n$ large enough one has
$$
n\begin{pmatrix} n \\ \alpha n \end{pmatrix}\leq 
e^{(\phi(\alpha)+\epsilon)n},
$$
and $\displaystyle \lim_{\alpha\to 0} \phi(\alpha)=0$.
\end{lemma}

The second fact follows from the definition of the local entropy and
Katok's entropy formula.

\begin{lemma}\label{lem:hmu-hloc}
Let $f\colon X\to X$ be a continuous self-map of a compact metric space
and $\mu$  an ergodic invariant Borel measure for $f$. Then for all $\epsilon>0$, 
$$
h(\mu,f)\leq h(\mu,f,\epsilon)+h_{{\rm loc}}(f,\epsilon).
$$
\end{lemma}

\begin{proof}[of the Theorem]
Let $\epsilon=\gamma/(2+\log(4\|f'\|_\infty+5))$.
One can choose $\delta>0$ such that:
\begin{gather*}
h_{{\rm top}}(C(f),f,\delta)<h_{{\rm top}}(C(f),f)+\epsilon\\
h_{{\rm loc}}(f,4\delta)<h_{{\rm loc}}(f)+\epsilon.
\end{gather*}
Let $C_*=\{k\delta\mid 1\leq k\leq [\delta^{-1}]\}$ and $C=C(f)\cup C_*$.
One has $r(\delta,n,C)\leq r(\delta,n,C(f))+\# C_*$ thus
$$
h_{top}(C,f,\delta)\leq h_{top}(C(f),f,\delta)<h_{top}(C(f),f)+\epsilon.
$$

There exists an integer $N_0$ such that, for all $n\geq N_0$,
$r(\delta,n,C)\leq e^{(h_{{\rm top}}(C(f),f)+\epsilon)n}$. Let $C_n$ be
a $(\delta,n)$-cover of $C$ of cardinality $r(\delta,n,C)$.

According to Lemma \ref{lem:combinatorial}, there exist two integers
$M, N$ such that $N\geq N_0$ and
\begin{equation}\label{eq:N}
\forall n\geq M,\ 
\frac{n}{N}\begin{pmatrix} n \\ 2n/N \end{pmatrix}<e^{\epsilon n}.
\end{equation}

Let $(\nu_m)_{m\geq 1}$ be a sequence of ergodic measures satisfying
the assumption of the Theorem. Observe that we can assume that
$h(\nu_m,\sigma)>h_{top}(C(f),f)$ for all integers $m\geq 1$.
By Proposition \ref{prop:high-levels}, 
$\lim_{m\to+\infty}\nu_m([{\mathcal D}_N])=0$. Fix
$m\geq M$ such that $\nu_m([{\mathcal D}_N])<\epsilon$ and define $\nu=\nu_m$ and
$\mu=\pi_*(\nu)$. Theorem~\ref{theo:bijection} says that
$h(\mu,f)=h(\nu,\sigma)$.

By the ergodic Theorem, for $\nu$-almost every $(\alpha_n)_{n\geq 0}
\in \Sigma_+({\mathcal D})$, one has
$$
\lim_{n\to+\infty}\frac{1}{n}\#\{0\leq k<n\mid \alpha_k\in{\mathcal D}_N\}=
\nu([{\mathcal D}_N])<\epsilon.
$$
Consequently, there exist a set $S_0\subset \Sigma_+({\mathcal D})$ and an integer
$T\geq M$ such that $\nu(S_0)>0$ and for all $(\alpha_n)_{n\geq 0}\in
S_0$ and $n\geq T$,
\begin{equation}\label{freqDN} 
\frac{1}{n}\#\{0\leq k<n\mid \alpha_k\in{\mathcal D}_N\}<\epsilon.
\end{equation}
Let $D\in{\mathcal D}$ such that $\nu(S_0\cap [D])>0$ and define $S=S_0\cap [D]$.
One has $\mu(\pi(S))\geq \nu(S)>0$.
We are going to bound $r(4\delta,n,\pi(S))$, which will give a 
bound on $h(\nu,\sigma)$.

Let $\alpha=(\alpha_n)_{n\geq 0}\in S$, $x=\pi(\alpha)$ and $n\geq T$.
We define a finite set $I=\{1,\dots,j\}$ and 
disjoint integer subintervals $[a_i,b_i),i\in I$ satisfying
the following properties.
\begin{enumerate}
\item $[a_i,b_i)\subset [-H(D),n)$ for all $i\in I$ (recall that
$H(D)$ is the height of $D$ in the graph ${\mathcal D}$, see section~\ref{secDiagram}).
\item $n_i=b_i-a_i>N$ for all $i\in I$.
\item $\#\left([0,n)\setminus \bigcup_{i\in I} [a_i,b_i)\right)<\epsilon n$.
\item There exists $z_i\in C_{n_i}$ such that 
$f^{a_i}(x)\in B(z_i,n_i,2\delta)$ for all $i\in I$.
\end{enumerate}

To define $I=\{1,\dots,j\}$ and the subintervals $[a_i,b_i)$, we set
$a_0=n$ and proceed inductively. Assume that $a_{i-1}$ is already defined.
Let $k$ be the largest integer such that $0<k\leq a_{i-1}$ and $\alpha_k
\not\in {\mathcal D}_N$. If there is no such $k$ then we stop here setting $j=i-1$.
Otherwise, we let $b_i=k$ and $a_i=b_i-H(\alpha_{b_i})$. Since 
$H(\alpha_{b_i})>N$ by choice of $k$, the induction ultimately ends.

We prove that these intervals have the stated properties.
The significant part of $\alpha_0=D$ is some $A_{-H(D)}\dots A_0$. 
By Lemma \ref{lem:follower} 
there exist $A_1,A_2,\ldots\in{\mathcal P}$, such that $\alpha_k$ is the equivalence 
class of $A_{-H(D)}\dots A_0A_1\dots A_k$ for each $k\geq 0$. This implies that 
$H(\alpha_k)\leq H(D)+k$. Therefore
$a_i=b_i-H(\alpha_{b_i})\geq -H(D)$; this is property (i).

By definition, $\alpha_{b_i}\not\in{\mathcal D}_N$, that is, $H(\alpha_{b_i})>N$.
Since $n_i=b_i-a_i=H(\alpha_{b_i})$, property (ii) holds.

Let $J=\{0\leq k<n\mid \alpha_k\in{\mathcal D}_N\}$. Equation (\ref{freqDN}) says
that $\# J<n\epsilon$. If $k$ satisfies $0<k\leq a_j$ or
$b_i<k\leq a_{i-1}$ for some $i\in I$, then $\alpha_k\in{\mathcal D}_N$ by
definition of $(b_i)_{i\in I}$. Therefore
$$
(0,a_j]\cup\bigcup_{i\in I} (b_i,a_{i-1}]\subset J.
$$
One has
$$
[0,n)\setminus\bigcup_{i\in I}[a_i,b_i)=
[0,a_j)\cup\bigcup_{i\in I}[b_i,a_{i-1}).
$$
Moreover, $\#[a,b)=\#(a,b]$, hence
$$
\#\left([0,n)\setminus\bigcup_{i\in I}[a_i,b_i)\right)=
\#\left((0,a_j]\cup\bigcup_{i\in I}(b_i,a_{i-1}]\right)\leq \# J<n\epsilon.
$$
This is property (iii).

Finally, we show that property (iv) holds.
Let $i\in I$ and let $A_{-p}\dots A_0$ be the significant part of
$\alpha_{a_i}$. Using Lemma \ref{lem:follower} there exist
$A_1,\dots,A_{n_i}\in{\mathcal P}$ such that $\alpha_{b_i}$ is the equivalence class
of $A_{-p}\dots A_0A_1\dots A_{n_i}$. But the significant part of 
$\alpha_{b_i}$ is some $B_{-n_i}\dots B_0$ (recall that 
$n_i=H(\alpha_{b_i})$). Therefore,
by definition of the equivalence, $A_0\dots A_{n_i}=B_{-n_i}\dots B_0$.
By definition of the significant part, we have
$$
f^{n_i-1}([B_{-n_i+1}\dots B_0]_f)\supsetneq f^{n_i}([B_{-n_i}\dots B_0]_f).
$$
By definition,
\begin{eqnarray*}
f^{n_i}([B_{-n_i}\dots B_0]_f)&=&\bigcap_{k=0}^{n_i} f^k(B_{-k})\\
 &=& f^{n_i}(B_{-n_i})\cap f^{n_i-1}([B_{-n_i+1}\dots B_0]_f),
\end{eqnarray*}
thus $f^{n_i}(B_{-n_i})\not\supset f^{n_i-1}([B_{-n_i+1}\dots B_0]_f)$,
which implies that
\begin{equation}\label{eq:boundary1}
f(B_{-n_i})\not\supset [B_{-n_i+1}\dots B_0]_f.
\end{equation}
In addition, $[B_{-n_i}\dots B_0]_f\not=\emptyset$ so that 
\begin{equation}\label{eq:boundary2}
f(B_{-n_i})\cap [B_{-n_i+1}\dots B_0]_f\not=\emptyset.
\end{equation}
$f$ is monotonic on  $\overline{B_{-n_i}}$ and by Lemma
\ref{lem:sig-part} the set $[B_{-n_i+1}\dots B_0]_f$ is an interval;
combining this with (\ref{eq:boundary1}) and
(\ref{eq:boundary2}), it follows that there exists $z\in\partial B_{-n_i}$
such that $f(z)\in[B_{-n_i+1}\dots B_0]_f$. In other words, 
$f^k(z)\in\overline{B_{-n_i+k}}$ for $k=0,\dots,n_i$. 

$H(\alpha_{b_i})=n_i$ and $b_i=a_i+n_i$, hence 
$f^{b_i-(n_i-k)}(x)=f^{a_i+k}(x)\in \overline{B_{-n_i+k}}$ 
for $k=0,\dots,n_i$ according to Lemma \ref{lemPasse}. Moreover 
the diameter of ${\mathcal P}$ is at most $\delta$ by the definition of $C$. 
Therefore $f^{a_i}(x)\in B(z,n_i,\delta)$.
Since $z\in \partial B_{-n_i}\subset C$, there exists $z_i\in C_{n_i}$
such that $z\in B(z_i,n_i,\delta)$. Thus $f^{a_i}(x)\in B(z_i,n_i,2\delta)$
and property (iv) is satisfied.

\medbreak
A {\em description} of $x$ up to time $n$ is a sequence of points 
$(x_k)_{0\leq k<n}$ such that
\begin{itemize}
\item $x_{a_i+k}=f^k(z_i)$ if $i\in I$ and $0\leq k<n_i$,
\item $x_k\in C_*$ and $|f^k(x)-x_k|\leq2\delta$ if $k\not\in
\bigcup_{i\in I}[a_i,b_i)$.
\end{itemize}
Notice that these conditions imply that $|f^k(x)-x_k|\leq 2\delta$ for
$0\leq k<n$.
Let us bound the number of distinct possible descriptions.

Firstly, $\#I\leq \frac{n+H(D)}{N}$ and, when $j=\#I$ is already fixed, there
are at most $\begin{pmatrix} n+H(D) \\ 2j\end{pmatrix}$ choices for the
positions of the integers $a_i,b_i\ (i\in I)$ in $[-H(D),n)$. Hence the total number
of choices of the intervals $[a_i,b_i)$ is bounded by
$$
\frac{n+H(D)}{N}\begin{pmatrix} n+H(D) \\ 2(n+H(D))/N
\end{pmatrix}< e^{\epsilon n},
$$
the inequality being implied by (\ref{eq:N}).

Secondly, for each $i\in I$, there are at most $\#C_{n_i}\leq
e^{(h_{{\rm top}}(C(f),f)+\epsilon)n_i}$ choices of $z_i\in C_{n_i}$ because
$n_i>N\geq N_0$. Thus the number of choices of $(z_i)_{i\in I}$ is
bounded by
$$
\prod_{i\in I} e^{(h_{{\rm top}}(C(f),f)+\epsilon)n_i}\leq e^{(h_{{\rm top}}(C(f),f)+\epsilon)n}.
$$

Thirdly, consider $k\in [0,n)\setminus \bigcup_{i\in I}[a_i,b_i)$. If
$k=0$ then the number of choices of $x_0$ is at most $\#C_*\leq \delta^{-1}$.
If $k>0$ then
\begin{eqnarray*}
|x_k-f(x_{k-1})|& \leq & |x_k-f^k(x)|+|f^k(x)-f(x_{k-1})|\\
 & \leq & 2\delta +|f^{k-1}(x)-x_{k-1}| \|f'\|_{\infty}\\
 & \leq & \delta \left(2\|f'\|_{\infty}+2\right).
\end{eqnarray*}
Thus the number of possible $x_k\in C_*$ is at most 
$4\|f'\|_{\infty}+5$ if the points $x_0,\dots, x_{k-1}$ are already chosen.
Moreover, 
$\#\left([0,n)\setminus \bigcup_{i\in I}[a_i,b_i)\right)<\epsilon n$
because the intervals $[a_i,b_i)$, $i\in I$ satisfy property (iii).
Therefore, the number of choices of $x_k$, 
$k\in [0,n)\setminus \bigcup_{i\in I}[a_i,b_i)$, is bounded by
$$
\delta^{-1}\left( 4\|f'\|_{\infty}+5\right)^{\epsilon n}.
$$

Finally, the number of distinct descriptions is at most
$$
N_d=\delta^{-1} e^{(h_{{\rm top}}(C(f),f)+\epsilon+\epsilon\log( 4\|f'\|_{\infty}+5))n}.
$$

If $x,y$ admit the same description then $|f^k(x)-f^k(y)|\leq 4\delta$
for all $0\leq k<n$. Therefore there exists a $(4\delta,n)$-cover of
$\pi(S)$ of cardinality at most $N_d$, that is,
$r(4\delta,n,\pi(S))\leq N_d$. But $h(\mu,f,4\delta)\leq
\limsup_{n\to +\infty} \frac{1}{n}\log r(4\delta,n,\pi(S))$ because
$\mu(\pi(S))>0$, hence
$$
h(\mu,f,4\delta)\leq h_{{\rm top}}(C(f),f)+\epsilon+\epsilon\log ( 4\|f'\|_{\infty}+5).
$$
According to Lemma \ref{lem:hmu-hloc} and the choice of $\delta$, one has
\begin{eqnarray*}
h(\mu,f) & \leq & h(\mu,f,4\delta)+h_{{\rm loc}}(f,4\delta)\\ 
&\leq &
h_{{\rm top}}(C(f),f)+h_{{\rm loc}}(f)+\epsilon\left(2+\log( 4\|f'\|_{\infty}+5)\right)\\
&\leq& h_{{\rm top}}(C(f),f)+h_{{\rm loc}}(f)+\gamma.
\end{eqnarray*}
Since $h(\nu,\sigma)=h(\mu,f)$, this concludes the proof.
\end{proof}

\section{Existence of maximal measures}

In this section, we prove the main Theorem by combining the results
of the previous sections.

\begin{theorem}\label{theo:main-theorem}
Let $f\colon [0,1]\to [0,1]$ be a $C^1$ map and $C(f)$ the critical
set of $f$. Assume that
$h_{{\rm top}}(f)>h_{{\rm top}}(C(f),f)+h_{{\rm loc}}(f)$. Then $f$ admits a maximal measure.
Moreover, the number of ergodic maximal measures is finite and, if $f$
is transitive, the maximal measure is unique.
\end{theorem}

\begin{proof}
Let $\epsilon>0$ such that $h_{{\rm top}}(f)>h_{{\rm top}}(C(f),f)+h_{{\rm loc}}(f)+\epsilon$.
Let $\delta>0$ be given by Theorem \ref{theo:eps-delta} with ${\mathcal D}$
the corresponding Markov diagram. By Theorem \ref{theo:bijection},
one has $h({\mathcal D})=h_{{\rm top}}(f)$. Suppose that $\Sigma_+({\mathcal D})$ has
no maximal measure. By Proposition \ref{prop:sequence-of-measures},
there exists a sequence of ergodic measures $(\nu_n)_{n\geq 1}$ such
that $h(\nu_n,\sigma)\to h({\mathcal D})=h_{{\rm top}}(f)$ and for all finite
subsets $F\subset {\mathcal D}$, $\nu_n([F])\to 0$.
By Theorem \ref{theo:eps-delta}, one has
$$
\limsup_{n\to+\infty}h(\nu_n,\sigma)
\leq h_{{\rm top}}(C(f),f)+h_{{\rm loc}}(f)+\epsilon<h_{{\rm top}}(f),
$$
which is a contradiction. Consequently $\Sigma_+({\mathcal D})$ has a maximal measure,
and so has $f$ by Theorem \ref{theo:bijection}, proving the first
claim of the Theorem.

Suppose now that there is a sequence $(\mu_n)_{n\geq 1}$ of distinct 
ergodic maximal
measures for $f$. Let $\nu_n$ be the ergodic measure on $\Sigma_+({\mathcal D})$ 
that corresponds to $\mu_n$ by $\pi$ (Theorem \ref{theo:bijection}).
The $\nu_n$ are distinct ergodic maximal measures, thus
for all finite subsets $F\subset {\mathcal D}$, one has $\nu_n([F])\to 0$
by Proposition \ref{prop:sequence-of-measures2}.
As previously, this leads to a contradiction by Theorem
\ref{theo:eps-delta}, proving the finiteness claim.

Finally, suppose that $f$ is transitive. Then by Lemma \ref{lem:transitive},
${\mathcal D}$ has a unique connected component of large entropy and therefore
admits at most one maximal measure by 
Theorem \ref{theo:existence-Markov}. 
Thus, $f$ has at most one maximal measure 
by Theorem \ref{theo:bijection}.
\end{proof}


\noindent
{\scshape J\'er\^ome Buzzi} -- Centre de Math{\'e}matiques de l'Ecole Polytechnique, U.M.R. 7640 du C.N.R.S., F-91128 Palaiseau cedex, France -- {\it E-mail address:} {\tt buzzi@math.polytechnique.fr}

\medskip\noindent
{\scshape Sylvie Ruette} -- Laboratoire de Math\'ematiques, Topologie et Dynamique, B\^at. 425, Universit\'e Paris-Sud, F-91405 Orsay cedex, France -- {\it E-mail address:} {\tt sylvie.ruette@math.u-psud.fr}

\end{document}